\documentclass[a4paper, 12pt,reqno]{amsart}
\usepackage[english]{babel}
\usepackage[utf8]{inputenc}
\usepackage{amsthm, amsmath, amssymb, amsrefs, enumerate, enumitem, mathtools, bbm, mathrsfs, array}
\usepackage{listings}

\usepackage{a4wide}
\binoppenalty=\maxdimen
\relpenalty=\maxdimen

\usepackage[T1]{fontenc}
\usepackage{lmodern}

\usepackage{tabularx}

\usepackage{hyperref}

\theoremstyle{plain}
\newtheorem{thm}{Theorem}[section]
\newtheorem*{thm*}{Theorem}
\newtheorem{prop}[thm]{Proposition}

\newtheorem{cor}[thm]{Corollary}
\theoremstyle{definition}

\theoremstyle{remark}

\newtheorem*{question}{Question}

\newcommand{\C}{\mathbb{C}}
\renewcommand{\H}{\mathbb{H}}
\newcommand{\Z}{\mathbb{Z}}

\newcommand{\N}{\mathbb{N}}
\newcommand{\R}{\mathbb{R}}

\newcommand{\re}{\textnormal{Re}}

\setlist{nosep}
\setlist{noitemsep}
\numberwithin{equation}{section}
\lstdefinelanguage{Sage}[]{Python}
{morekeywords={False,sage,True},sensitive=true}
\newcolumntype{M}[1]{>{\centering\arraybackslash}m{#1}}

\title[Crank equidistribution and $(k,j)$-overlined partitions]{Crank equidistribution and $(k,j)$-overlined partitions}
\author{Adithya Chakravarthy}
\address{Department of Mathematics, University of Toronto
	Bahen Centre, 40 St. George St., Toronto, ON}
\email{adithya.chakravarthy@mail.utoronto.ca}
\author{Joshua Males}
\address{School of Mathematics, University of Bristol, Bristol, BS8 1TW, UK, and the Heilbronn Institute for Mathematical Research, Bristol, UK.}
\email{joshua.males@bristol.ac.uk}
\author{Shuyang Shen}
\address{Department of Mathematics, University of Toronto
	Bahen Centre, Room 6290
	40 St. George St., Toronto, ON}
\email{shuyang.shen@mail.utoronto.ca}

\begin{document}

\begin{abstract}
In a paper published in 2023, Wagner introduced and studied Jacobi forms with complex multiplication, and gave several applications. One such application was in constructing a new doubly-infinite family of partition-theoretic objects, called $(k,j)$-coloured overpartitions and labelled by $\overline{p}_{k,j}$, and using the Jacobi forms to construct crank functions which explain the Ramanujan-type congruences satisfied by $\overline{p}_{k,j}$. In this note, we give an asymptotic formula for the number of $(k,j)$-coloured overpartitions and prove that any crank constructed by Wagner is asymptotically equidistributed on arithmetic progressions, following several recent papers in the literature.
\end{abstract}

\subjclass[2020]{}


\thanks{ }

\maketitle

\section{Introduction}

A partition of a positive integer $n$ is a non-increasing list $\lambda = (\lambda_1,\dots,\lambda_s)$ such that $\sum_j \lambda_j=n$. Let $p(n)$ denote the number of partitions of $n$. The function $p(n)$ and its variants are some of the most widely studied objects at the interface of number theory and combinatorics. In particular, they were the famous object of study when Hardy and Ramanujan \cite{HardyRamanujan} developed the now-ubiquitous Circle Method and proved the stunning asymptotic formula
\begin{align}\label{eqn: asymp p}
		p(n)\sim \frac{1}{4\sqrt{3}n} e^{\pi \sqrt{\frac{2n}{3}}}, \qquad n \to \infty.
\end{align}

Using techniques in $q$-series, Ramanujan also proved that $p(n)$ satisfies certain congruences modulo $5,7,$ and $11$. In particular
\begin{align*}
	p(5n+4) \equiv 0 \pmod{5}, \qquad
	p(7n+5) \equiv 0 \pmod{7}, \qquad
	p(11n+6) \equiv 0 \pmod{11}.
\end{align*}
However, his proof was unable to combinatorially explain why the Ramanujan congruences hold. In order to attempt to explain the congruences, Dyson \cite{Dyson1944} famously introduced a statistic on partitions called the rank, defined as the largest part in the partition minus the number of parts. Atkin and Swinnerton-Dyer were able to use Dyson's rank to explain the Ramanujan congruences modulo $5$ and $7$ in 1954 \cite{ASD}.  However, the rank is unable to explain the modulo $11$ congruence, and so Dyson conjectured that there should exist a further statistic on partitions which explains all three congruences simultaneously. He dubbed this unknown statistic the crank.

It took until the late 1980s before the crank statistic was found by Garvan and Andrews \cite{AGBAMS, GarvanTAMS} who defined
	\begin{equation}\label{def_crank}
	\begin{aligned}\text{crank}(\lambda) &:= \begin{cases} \text{largest part of } \lambda & \text{if 1 is not a part of } \lambda, \\
			\mu(\lambda) - o(\lambda) & \text{if 1 is a part of } \lambda, \end{cases}
\end{aligned}\end{equation}	
where $\mu(\lambda)$ denotes the number of parts of $\lambda$ strictly larger than the number of $1$s in $\lambda$, and $o(\lambda)$ denotes the number of $1$s in $\lambda$. They used the crank to combinatorially explain all three of Ramanujan's congruences for $p(n)$.

Partitions have also been shown to satisfy many interesting inequalities. For example, DeSalvo and Pak \cite{DeSalvo} showed that $p(n)$ satisifies the log-concavity inequality
\begin{align*}
	p(n)^2 > p(n+1)p(n-1),
\end{align*}
for all $n\geq 26$. Log concavity can also be viewed as the order two Tur\'{a}n inequality. The higher-order Tur\'{a}n inequalities for $p(n)$ (and its variants) have been studied in the literature since the major breakthrough paper \cite{GORZ}, which proved that the Jensen polynomial associated to the Fourier coefficients of any (weakly holomorphic) modular forms is asymptotically hyperbolic, in turn meaning that the Fourier coefficients asymptotically satisfy all higher order Tur\'{a}n inequalities. There are a plethora of further inequalities satisfied by $p(n)$, for example the Bessenrodt-Ono inequality \cite{BO}
\begin{align*}
	p(n_1)p(n_2) \geq p(n_1+n_2)
\end{align*}
for all $n_1,n_2 \in \N$ such that $n_1+n_2 > 8$. In \cite{Bal}, Bal, Haraldson, Thompson and the second author showed that any sequence arising from functions which satisfy the hypotheses of Proposition \ref{WrightCircleMethod} (which are not necessarily modular forms) also satisfy a Bessenrodt-Ono type inequality as well as all higher-order Tur\'{a}n inequalities. We make use of these results to immediately conclude similar statements on the objects of our study.

There are many variants of partitions in the literature, and here we focus on a recent family introduced by Wagner \cite{Wag} called $(k,j)$-coloured overpartitions, where $0 < j \leq k$. These are coloured partitions of $n$ where the first occurrence of any number of $j$ of the colours may be overlined, and we denote their count by $\overline{p}_{k,j}(n)$. In what follows, we imitate the three themes of study for $p(n)$ highlighted above for this new family of partitions.

Of particular importance to us is the generating function (the reason we call this $H(1;q)$ will soon become apparent)
\begin{align}\label{eqn: gen fn}
H(1;q) \coloneqq \sum_{n\geq 0} \overline{p}_{k,j}(n) q^n = \prod_{n \geq 1} \frac{\left(1+q^n\right)^j}{\left( 1-q^n \right)^k},
\end{align}
with $q \coloneqq e^{2\pi i z}$ for $z \in \H$, the upper half-plane. Our first main result imitates the asymptotic formula \eqref{eqn: asymp p} but for this new family of partitions, where we use Wright's variant of the Circle Method in order to obtain an error term.

\begin{thm}\label{Thm: main1}
	Let $\overline{p}_{k,j}(n)$ count the number of $(k,j)$-coloured overpartitions of $n$. Then as $n\to \infty$ we have
	\begin{align*}
		\overline{p}_{k,j}(n) = \frac{  (2k+j)^{\frac{k+1}{4}} }{  2^{\frac{2k+j+3}{2}}  3^{\frac{k+1}{4}}  n^{\frac{k+3}{4}}   } e^{\pi \sqrt{\frac{(2k+j)n}{3}}} \left(1+O\left(n^{-\frac{1}{2}} \right)\right) \eqqcolon C(k,j;n) \left(1+O\left(n^{-\frac{1}{2}} \right)\right).
	\end{align*}
\end{thm}

We illustrate these numerics in the Table \ref{Tab: sampleratios} where, for several choices of $(k,j)$ and $n$, we give the ratio $\frac{\overline{p}_{k,j}(n)}{C(k,j;n)}$.

\begin{table}[h]
	\centering
	\begin{tabular}{c |c|c|c|c}
		${\bf (k, j)}$ & ${\bf n = 100}$ & $ {\bf n = 1000}$ & ${\bf n = 2000}$ & ${\bf n = 5000}$\\
		\hline
		$(1, 1)$  & 0.835\dots & 0.943\dots & 0.959\dots & 0.974\dots \\
		$(2, 1)$  & 0.782\dots & 0.923\dots & 0.945\dots & 0.964\dots \\
		$(3, 1)$  & 0.735\dots & 0.904\dots & 0.931\dots & 0.956\dots \\
		$(3, 2)$  & 0.732\dots & 0.903\dots & 0.930\dots & 0.955\dots \\
		$(5, 3)$  & 0.653\dots & 0.870\dots & 0.906\dots & 0.939\dots
	\end{tabular}
	\caption{Sample values for $\frac{\overline{p}_{k,j}(n)}{C(k,j;n)}$ for several choices of $(k,j)$ and $n$.}
	\label{Tab: sampleratios}
\end{table}

In Theorems 4.17 and 4.18 of \cite{Wag}, using the theory of complex multiplication for modular forms (lifted to Jacobi forms), Wagner showed that $(k,j)$-coloured overpartitions satisfy many Ramanujan-type congruences, i.e. congruences of the shape
\begin{align}\label{eqn: Ram cong}
	\overline{p}_{k,j}(\ell n + \delta_{k,j,\ell} ) \equiv 0 \pmod{\ell}
\end{align}
for primes $\ell \geq 5$ and $\delta_{k,j,\ell}  \in \N$. Moreover, Wagner showed how to construct many crank statistics (in the guise of their generating functions) which combinatorially explain a given instance of \eqref{eqn: Ram cong}. To state these results in a more convenient way, we use the following notation where $\zeta$ is a root of unity
\begin{align*}
	F_1(\zeta;q) \coloneqq \prod_{n \geq 1} \left(1-\zeta q^n \right),
\end{align*}
along with shorthand notation $\pm$inside functions to indicate that one should take the product of the function with both choices.

 He gave a general procedure of how to produce crank generating functions in terms of $F_1$ and functions labelled $\widetilde{\phi}_r(\zeta_1,\zeta_2;q)$ which are certain Jacobi forms with prefactors removed. In general, the two-variable crank generating functions have the form (see \cite[Section 4.3]{Wag})
\begin{align*}
	\frac{\widetilde{\phi}_r\left(\zeta^a; \zeta^b; q\right)}{ \left[F_1\left(1;q \right) F_1\left(\zeta^{\pm a_1};q\right) F_1 \left(\zeta^{\pm a_2};q\right) \cdots F_1\left( \zeta^{\pm a_{\frac{\ell-1}{2}}}; q \right) \right]^j},
\end{align*}
with $r=4,6,10$ and where $(\pm a_1, \pm a_2,\dots, \pm a_{\frac{\ell-1}{2}})$ form a complete set of residues modulo $\ell$. Wagner requires some cancellation between the numerator and denominator in order to call the statistic defined by such a generating function a crank \footnote{We assume that the numerator does not completely cancel the denominator.}. This generating function inherently defines a crank statistic $c$. Moreover, since we have some choice in the cancellation that occurs, let us assume that choosing a fixed crank statistic $c$, we may rewrite our two-variable crank generating function (after cancellation) as
\begin{align}\label{eqn: crank as F_1}
H^c(\zeta;q) \coloneqq \prod_{j=1}^r \frac{  F_1(1; q^{b_j})^{c_j}   }{F_1(\zeta^{\pm d_j};q)^{f_j}},
\end{align}
with $c_j,d_j,f_j\in \N_0$ and $b_j, r \in \N$.

In particular, $H^c$ is completely in terms of the functions $F_1$ which were studied by Bringmann, Craig, Ono, and the second author in \cite{BCMO}. Also recall that since $c$ defines a crank statistic, we must have that $H^c(1;q) = \sum_{n \geq 0} \overline{p}_{k,j}(n) q^n$.

For example, Wagner proved that for $(3,2)$-coloured overpartitions a crank that combinatorially explains the Ramanujan-type congruence modulo $7$ has the two-variable generating function
\begin{align*}
	\prod_{n \geq 1} \frac{\left(1-q^n \right) \left(1+q^{n}\right)^2}{  \left( 1-\zeta^{\pm 1} q^n \right) \left( 1-\zeta^{\pm 2} q^n \right)  } = \frac{ F_1(1;q^2)^2}{F_1(1;q) F_1(\zeta^{\pm 1};q) F_1(\zeta^{\pm 2};q) }.
\end{align*}

Let $c$ be any crank statistic of Wagner as defined above and assume that $H^c(\zeta;q)$ takes the form \eqref{eqn: crank as F_1} such that $\gcd(d_1,d_2,\dots,d_r)=1$. Our second main theorem is the following equidistribution result (for any modulus $b >1$), which combined with Theorem \ref{Thm: main1} gives the asymptotic distribution for all residue classes in any modulus. Let $ \overline{p}^c_{k,j}(a,b;n)$ count the number of of $(k,j)$-coloured overpartitions of $n$ with fixed crank statistic $c$ equivalent to $a \pmod{b}$.

\begin{thm}\label{Thm: Main2}
	As $n \to \infty$ we have
	\begin{align*}
		\overline{p}^c_{k,j}(a,b;n) \sim \frac{1}{b} \overline{p}_{k,j}(n).
	\end{align*}
Moreover, Theorem \ref{Thm: main1} then implies that
\begin{align*}
	\overline{p}^c_{k,j}(a,b;n) = \frac{C(k,j;n)}{b} \left(1+O\left(n^{-\frac{1}{2}} \right)\right).
\end{align*}
\end{thm}

The fact that we require $\gcd(d_1,d_2,\dots,d_r)=1$ is purely for convenience. One may check that our techniques hold in the case where the $\gcd >1$, and that one simply needs to take care to add further terms in the multisection that contribute to the main term asymptotic correctly (similar to e.g. \cite[Theorem 1.4]{BCMO} or \cite[Theorem 1.5]{CCM}).

Given Theorems \ref{Thm: main1} and \ref{Thm: Main2} we are also able to immediately conclude asymptotic inequalities that the coefficients $\overline{p}^c_{k,j}(n)$ and $\overline{p}^c_{k,j}(a,b;n)$ satisfy, in a similar vein to various other papers in the literature. In particular, using  \cite[Corollaries 3.2 and 3.3]{CCM} and \cite[Theorem 1.2]{Bal} we obtain the following corollaries. The second is also a direct consequence of the generating function for $\overline{p}_{k,j}(n)$ being (essentially) a modular form, and an application of \cite[Theorem 3]{GORZ} with Theorem \ref{Thm: Main2}.

\begin{cor}
		For large enough $n_1$ and $n_2$, we have that
	\begin{align*}
		\overline{p}^c_{k,j}(a,b;n_1)\overline{p}^c_{k,j}(a,b;n_2) > \overline{p}_{k,j}(a,b;n_1+n_2),
	\end{align*}
	and 
	\begin{align*}
		\overline{p}_{k,j}(n_1) \overline{p}_{k,j}(n_2) > \overline{p}_{k,j}(n_1+n_2).
	\end{align*}
\end{cor}

\begin{cor}
	For large enough $n$, both $\overline{p}_{k,j}(n)$ and $\overline{p}^c_{k,j}(a,b;n)$ satisfy all higher-order Tur\'{a}n inequalities.
\end{cor}

Moreover, we also immediately obtain the following corollary via \cite[Theorem 1.4]{Wagner}, since the main term asymptotic for both $\overline{p}_{k,j}(n)$ and $\overline{p}^c_{k,j}(a,b;n)$ arise as Taylor coefficients of a suitable modular form (see \cite{Wagner} for more background on Laguerre inequalities).
\begin{cor}
		For large enough $n$, both $\overline{p}_{k,j}(n)$ and $\overline{p}^c_{k,j}(a,b;n)$ satisfy all higher-order Laguerre inequalities.
\end{cor}

\section{Preliminaries}
Here we recall some preliminaries required for the rest of the paper.

\subsection{Multisections}\label{Sec: multisections}
Fix a crank statistic $c$. Let
 \begin{align*}
	H^c(\zeta;q) \coloneqq \sum_{\substack{n \geq 0 \\ m \in \Z}} \overline{p}^c_{k,j}(m,n) \zeta^m q^n,
\end{align*}
where $\overline{p}^c_{k,j}(m,n)$ denotes the number of $(k,j)$-coloured overpartitions of $n$ with crank statistic $c$ precisely equal to $m$. 
A standard computation using orthogonality of roots of unity (sometimes called the multisection of the formal power series) gives
\begin{align}\label{eqn: splitting}
	H^c(a,b;q) \coloneqq	\sum_{n \geq 0} \overline{p}^c_{k,j}(a,b;n) q^n = \frac{1}{b} \sum_{n \geq 0} \overline{p}_{k,j}(n) q^n + \frac{1}{b} \sum_{k=1}^{b-1} \zeta_b^{-ak} 	H^c(\zeta_b^k;q).
\end{align}

To prove equidistribution results, we see that we need the first term on the right-hand side to be asymptotically dominant. This is the central idea used by Cesana, Craig, and the second author in \cite{CCM}, which provides a framework for proving equdisitribution results, building on the examples of (non)-equidistribution given in \cite{BCMO}. In the present paper, we follow similar ideas to those in \cite{BCMO,CCM} in proving Theorem \ref{Thm: Main2}.

\subsection{Asymptotics of infinite $q$-products} 
We require several estimates on the asymptotics of infinite $q$-products which appear in the generating functions in this paper. We begin by recalling the classical transformation formula of the Dedekind $\eta$-function, see e.g. see 5.8.1 of \cite{CS}, which implies that for $q = e^{2\pi i z}$ we have
\begin{align}\label{Eqn: transform usual pochham}
	(q;q)_\infty = z^{-\frac 12} e^{\frac{\pi}{12}\left( z - \frac{1}{z} \right)} (q_1;q_1)_\infty,
\end{align}
where $q_1 := e^{-\frac{2\pi i}{z}}$. This gives the well-known bound (for $|y|\le Mx$, as $z\to0$)
\begin{align}
\left(e^{-z};e^{-z} \right)_\infty^{-1} = \sqrt{\frac{z}{2\pi}} e^{\frac{\pi^2}{6z}} (1+O(|z|)).
\end{align}
We also require the classical bound on minor arcs (i.e. for $|y|\ge Mx$ as $z\to 0$), see e.g. \cite[Lemma 3.5]{BD}
\begin{align}\label{eqn: q-poch minor}
	\left|\left(e^{-z} ; e^{-z} \right)_\infty^{-1}\right| \le x^{\frac12} e^{\frac{\pi^2}{6x}-\frac{\mathcal{C}}{x}},
\end{align}
for some constant $C>0$. 

Let Lerch's transcendent be denoted by
\begin{align*}
	\Phi(z,s,a)\coloneqq \sum_{n=0}^\infty \frac{z^n}{(n+a)^s},
\end{align*}
and for $0\leq \theta < \frac{\pi}{2}$ define the domain $D_{\theta} \coloneqq \left\{ z=re^{i\alpha} \colon r \geq 0 \text{ and } |\alpha| \leq \theta \right\}$, which cuts out a cone in the right half-plane. Throughout, the Gamma function is defined as usual by $\Gamma(x)\coloneqq \int_0^\infty t^{x-1} e^{-t} dt $, for $\operatorname{Re}(x)>0$.
Then \cite[Theorem 2.1]{BCMO} is the following result, which will enable us to estimate $F_1$ on major arcs when we apply Wright's Circle Method.
\begin{thm} \label{Thm: asymptotics F}
	For $b\geq 2$, let $\zeta$ be a primitive $b$-th root of unity. Then as $z \to 0$ in $D_\theta$, we have 
	\begin{align*}
		F_{1}\left(\zeta;e^{-z}\right)  =\frac{1}{\sqrt{1-\zeta}} \, e^{-\frac{\zeta\Phi(\zeta,2,1)}{z}}\left( 1+O\left(|z|\right) \right).
	\end{align*}
\end{thm}

\subsection{Wright's Circle Method}

In order to obtain the asymptotic behaviour of our coefficients, we make use of Wright's Circle Method. The original Circle Method of Hardy and Ramanujan (extended by Rademacher) is extremely powerful, but requires a lot of input information and technical work. Wright developed an easier-to-use style of Circle Method, which requires less work, but trades this off against loss of information (more precisely, only obtaining an error term, and losing the possibility of an exact formula). One uses Cauchy's residue theorem to write the Fourier coefficients as a contour integral of the generating function over a circular contour $C$ of radius less than one. We pick a radius such that $C$ tends to the unit circle as $n\to\infty$. One then splits $C$ into arcs where the generating function has relatively large (resp.\@ small) asymptotic growth, called the major (resp.\@ minor) arcs. On the major arcs, we use asymptotic techniques to closely approximate the behaviour of the generating function, while on the minor arcs we bound more crudely. In \cite{BCMO}, following work of Ngo and Rhoades \cite{NgoRhoades}, the following result based on Wright's Circle Method was proved.

\begin{prop}[Proposition 4.4 of \cite{BCMO}] \label{WrightCircleMethod}
	Suppose that $F(q)$ is analytic for $q = e^{-z}$ where $z=x+iy \in \C$ satisfies $x > 0$ and $|y| < \pi$, and suppose that $F(q)$ has an expansion $F(q) = \sum_{n=0}^\infty c(n) q^n$ near 1. Let $c,N,M>0$ be fixed constants. Consider the following hypotheses:
	
	\begin{enumerate}[leftmargin=*]
		\item[\rm(1)] As $z\to 0$ in the bounded cone $|y|\le Mx$ (major arc), we have
		\begin{align*}
			F(e^{-z}) = z^{B} e^{\frac{A}{z}} \left( \sum_{j=0}^{N-1} \alpha_j z^j + O\left(|z|^N\right) \right),
		\end{align*}
		where $\alpha_s \in \C$, $A\in \R^+$, and $B \in \R$. 
		
		\item[\rm(2)] As $z\to0$ in the bounded cone $Mx\le|y| < \pi$ (minor arc), we have 
		\begin{align*}
			\lvert	F(e^{-z}) \rvert \ll e^{\frac{1}{\mathrm{Re}(z)}(A - \kappa)}.
		\end{align*}
		for some $\kappa\in \R^+$.
	\end{enumerate}
	If  {\rm(1)} and {\rm(2)} hold, then as $n \to \infty$ we have for any $N\in \R^+$ 
	\begin{align*}
		c(n) = n^{\frac{1}{4}(- 2B -3)}e^{2\sqrt{An}} \left( \sum\limits_{r=0}^{N-1} p_r n^{-\frac{r}{2}} + O\left(n^{-\frac N2}\right) \right),
	\end{align*}
	where $p_r := \sum\limits_{j=0}^r \alpha_j c_{j,r-j}$ and $c_{j,r} := \dfrac{(-\frac{1}{4\sqrt{A}})^r \sqrt{A}^{j + B + \frac 12}}{2\sqrt{\pi}} \dfrac{\Gamma(j + B + \frac 32 + r)}{r! \Gamma(j + B + \frac 32 - r)}$. 
\end{prop}

This result means that one need only verify the two hypotheses in order to obtain the asymptotic behaviour of the coefficients at hand.

\section{Proofs of the main theorems}

In this section we prove our main results. We begin with the asymptotic behaviour of $(k,j)$-coloured overpartitions.

\begin{proof}[Proof of Theorem \ref{Thm: main1}]
First note that
 \begin{align}\label{p_k,j as pochhammer}
	\sum_{n \geq 0} \overline{p}_{k,j}(n) = \frac{(q^2;q^2)_\infty^j}{(q;q)_\infty^{k+j}} 
\end{align}
Then using \eqref{Eqn: transform usual pochham} for $q = e^{-z}$ with $|y|<Mx$ we have that
\begin{align*}
	\frac{(q^2;q^2)_\infty^j}{(q;q)_\infty^{k+j}}  = 2^{-\frac{j}{2}} \left(\frac{z}{2\pi}\right)^{\frac{k}{2}} \exp\left( \frac{\pi^2 (2k+j)}{12z} \right) \left(1+O\left(|z|\right)\right).
\end{align*}
Moreover, recall the classical bound on the minor arc in \eqref{eqn: q-poch minor}. Using this, it easily follows that \eqref{p_k,j as pochhammer} has a dominant asymptotic on the major arc (i.e. $|y|<Mx$). 

Therefore we may apply Proposition \ref{WrightCircleMethod} to \eqref{WrightCircleMethod} with the parameters $A = \frac{(2k+j)\pi^2}{12}$, $B= \frac{k}{2}$ and $\alpha_0 = \frac{1}{2^{\frac{j}{2}}(2\pi)^{\frac{k}{2}}}$. After tidying up coefficients, this yields the asymptotic
\begin{align*}
	\overline{p}_{k,j}(n) = \frac{  (2k+j)^{\frac{k+1}{4}} }{  2^{\frac{2k+j+3}{2}}  3^{\frac{k+1}{4}}  n^{\frac{k+3}{4}}   } e^{\pi \sqrt{\frac{(2k+j)n}{3}}} \left(1+O\left(n^{-\frac{1}{2}} \right)\right)
\end{align*}
as claimed.
\end{proof}

We next turn to the proof of Theorem \ref{Thm: Main2}.

\begin{proof}[Proof of Theorem \ref{Thm: Main2}]
	We start by using the discussion on multisections of power series given in Section \ref{Sec: multisections} to rewrite
	\begin{align*}
			H^c(a,b;q) = \sum_{n \geq 0} \overline{p}^c_{k,j}(a,b;n) q^n = \frac{1}{b} \sum_{n \geq 0} \overline{p}_{k,j}(n) q^n + \frac{1}{b} \sum_{k=1}^{b-1} \zeta_b^{-ak} 	H^c(\zeta_b^k;q).
	\end{align*}
Our aim is to show that on both the major and minor arcs we have 
\begin{align*}
	H^c(\zeta_b^k;q) < H^c(1;q).
\end{align*}
In turn, this means that $H^c(1;q)$ dominates the asymptotic growth of  $H^c(a,b;q)$, and since it does not depend on the residue class $a \pmod{b}$ we have asymptotic equidistribution. Furthermore, we know that $H^c(1;q)$ satisfies the hypotheses on Proposition \ref{WrightCircleMethod} and so we obtain the full asymptotic behaviour by applying Theorem \ref{Thm: main1}.

To see this, first note that by equation \eqref{eqn: crank as F_1} we have that
\begin{align}\label{eqn: multi}
	H^c(\zeta;q) \coloneqq \prod_{j=1}^r \frac{  F_1(1; q^{b_j})^{c_j}   }{F_1(\zeta^{\pm d_j};q)^{f_j}},
\end{align}
with $c_j,d_j,f_j\in \N_0$ and $b_j \in \N$. Recall that we assume that $\gcd(d_1,d_2,\dots,d_r)=1$. This in turn implies that for any choice of modulus $b$, there are no choices of $k$ such that $H(\zeta_b^k;q) = H(1;q)$, and so we cannot have two terms each contributing to the main asymptotic term\footnote{If we drop the condition on requiring the $\gcd$ to be $1$, one would need to be more careful in analysing which terms contribute to the main asymptotic term here. We would then get non-uniform asymptotics, but what is referred to in the literature as ``essentially'' equidistributed statistics.}. That is, there is always at least one term in $H^c(\zeta_b^k;q)$ of the form $F_1(\xi;q)^{-1}$ for $\xi$ a root of unity not equal to $1$.

We inspect this term in closer detail. By Theorem \ref{Thm: asymptotics F} we have that
\begin{align*}
			F_{1}\left(\xi;e^{-z}\right)^{-1}  =\sqrt{1-\xi} e^{\frac{\xi\Phi(\xi,2,1)}{z}}\left( 1+O\left(|z|\right) \right).
\end{align*}
Note that $\xi \Phi(\xi,2,1) = \operatorname{Li}_2(\xi)$, where $\operatorname{Li}_2$ is the usual dilogarithm function.  Then, in order for this term to give an exponentially smaller contribution to the asymptotics than its counterpart $F_1(1;q)^{-1}$ arising from the $k=0$ term of \eqref{eqn: multi}, we need that 
\begin{align*}
	\re\left(\operatorname{Li}_2(\xi)\right) < \frac{\pi^2}{6}
\end{align*}
for any root of unity $\xi \neq 1$. This follows from the fact that (see e.g. \cite[Page 11]{Zag})
\begin{align*}
		\re\left(\operatorname{Li}_2(e^{ i \theta})\right)  = \frac{\pi^2}{6} - \frac{\theta (2\pi -\theta)}{4}
\end{align*}
 with $0 \leq \theta \leq 2\pi$.
 
 Then it is clear that on the major and minor arcs we must have that $H^c(\zeta_b^k;q) < H^c(1;q)$ as desired. An application of Theorem \ref{Thm: main1} then implies that for any crank statistic $c$ we have
 \begin{align*}
 	\overline{p}^c_{k,j}(a,b;n) = \frac{C(k,j;n)}{b} \left(1+O\left(n^{-\frac{1}{2}} \right)\right),
 \end{align*}
as claimed.
\end{proof}

\section{Further discussion}
\subsection{The Laguerre-Polya class}

The family of functions whose coefficients are associated to Jensen polynomials that are hyperbolic is known as the Laguerre-Polya class of functions. In \cite{Wagner}, Wagner introduced a new family of related functions which he called the shifted Laguerre-Polya class. These can be characterised by their associated Jensen polynomials being asymptotically hyperbolic - precisely the condition which the authors of \cite{GORZ} used to show that the coefficients of weakly holomorphic modular forms are asymptotically hyperbolic, and which was extended in \cite{Bal} to include all functions (not necessarily modular) satisfying the hypotheses in Proposition \ref{WrightCircleMethod}.

It is thus clear that in the present paper, we give new infinite families of functions lying in the shifted Laguerre-Polya class of functions, arising from the coefficients $\overline{p}_{k,j}(n)$ and $\overline{p}_{k,j}(a,b;n)$, which we record in the following corollary.

\begin{cor}
	Choose $(k,j)$ such that $0<j\leq k$ and a crank function $c$ as above. For any $b \geq 2$ and $a\pmod{b}$, all of the functions
	\begin{align*}
		\sum_{n \geq 0} \frac{\overline{p}_{k,j}(n)}{n!} x^n, \qquad \sum_{n \geq 0} \frac{\overline{p}^c_{k,j}(a,b;n)}{n!} x^n
	\end{align*}
	lie in the shifted Laguerre-Polya class of functions.
\end{cor}

It is natural to search for the lower bound (which must exist) above which the functions actually lie in the ``usual'' Laguerre-Polya class, and so we pose this as a question to the interested reader.
\begin{question}
	Can one obtain the explicit lower bound $N$ such that for all $n > N$ the Jensen polynomial associated to $\overline{p}_{k,j}(n)$ or $\overline{p}^c_{k,j}(a,b;n)$ is hyperbolic?
\end{question}

Furthermore, it is clear that for a generating function $\sum_{n\geq 0} c(a,b;n) q^n$ (which may not be a modular form) whose first term in the multisection (similar to \eqref{eqn: multi}) dominates the asymptotic and is modular will lie in the shifted Laguerre-Polya class. It is perhaps possible that one may exploit this knowledge to obtain very strong estimates for the modular term (using e.g. the full Hardy-Ramanujan-Rademacher Circle Method) and aim to produce stronger estimates for the remaining (non-modular) terms. In the case of the functions studied here, this would rely on answering the following.
\begin{question}
	Can one obtain stronger bounds on the further terms in the multisection? In particular, can one obtain stronger bounds on products and quotients of $F_1(\zeta;q)$?
\end{question}

\subsection{Cyclotomic polynomials}

Wagner \cite{Wag} produced the crank functions $c$ central to this paper by using the fact that to be equidistributed (not just asymptotically) on an arithmetic progression modulo $\ell$, one must have divisibility by cyclotomic polynomials $\Phi_\ell$ as formal Laurent series. This observation has also been used by other authors in influential work on proving the unimodality of the rank function of Dyson (that is, a famous conjecture of Stanton) \cite{BMR}, as well as producing further infinite families of crank functions for ordinary partitions \cite{BGRT}, and in giving new proofs of Ramanujan-type congruences for other partition-theoretic objects \cite{FMR}.

That the functions $\overline{p}_{k,j}(n)$ satisfy many Ramanujan-type congruences was proved by Wagner using the theory of complex multiplication for modular forms, and as alluded to in the introduction he produced many examples of cranks which explain the congruences. It would be interesting to take such cranks and understand the combinatorial explanation more deeply. 
\begin{question}
	Can one prove the Ramanujan-type congruences satisfied by $\overline{p}_{k,j}(n)$ on the level of partitions? That is, give a precise combinatorial map which shows the congruences explicitly.
\end{question}

Given that the cyclotomic polynomials divide the generating function for the cranks as Laurent series, their quotient is also a Laurent series with integer coefficients. These coefficients would then appear to contain some arithmetic information which would extremely interesting to describe - in particular since we here have infinite many cranks explaining Ramanujan-type congruences. Following similar questions posed in \cite{FMR}, we ask the following, where $H^c(\zeta;q) \coloneqq \sum_{n \geq 0} H_n^c(\zeta) q^n$.
\begin{question}
	Fix a choice of crank $c$ which explains a given Ramanujan-type congruence on the progression $\ell n+b$. What is the combinatorial interpretation of the coefficients of the Laurent series $H_{\ell n+b}^c(\zeta)/\Phi_\ell(\zeta)$?
\end{question}


\begin{bibsection}
\begin{biblist}
	
				\bibitem{AGBAMS} G.E. Andrews and F.G. Garvan, \textit{Dyson's crank of a partition,} Bull. Amer. Math. Soc. (18) no.\ 2 (1988), 167-171.
				
				\bibitem{ASD} A. O. L. Atkin and P. Swinnerton-Dyer, \textit{Some properties of partitions}, Proc. London Math. Soc. (3) 4 (1954), 84--106.
				
				\bibitem{Bal} J. Bal, F. Haraldson, J. Males, I. Thompson, \textit{Jensen polynomials associated with Wright's circle method: Hyperbolicity and Turán inequalities}, preprint.
				
				
					\bib{BO}{article}{
					author = {{Bessenrodt}, C.},
					author={{Ono}, K.},
					TITLE = {Maximal multiplicative properties of partitions}
					JOURNAL = {Ann. Comb.},
					FJOURNAL = {Annals of Combinatorics},
					VOLUME = {20},
					YEAR = {2016},
					NUMBER = {1},
					PAGES = {59--64},
					ISSN = {0218-0006},
				}
				
	\bibitem{BCMO} K. Bringmann, W. Craig, J. Males, and K. Ono, \textit{Distributions on partitions arising from Hilbert schemes and hook lengths}, Forum Math. Sigma, 10, E49. 
	
		\bibitem{BD} K. Bringmann and J. Dousse, \textit{On Dyson's crank conjecture and the uniform asymptotic behavior of certain inverse theta functions}, Trans. Amer. Math. Soc. 368 (2016), no. 5, 3141--3155.
	
		\bibitem{BGRT} K. Bringmann, K. Gomez, L. Rolen, and Z. Tripp, \textit{Infinite families of crank functions, {S}tanton-type conjectures, and unimodality}, Res. Math. Sci., 9 (2022), no. 3, Paper No. 37.
		
	\bibitem{BMR} K. Bringmann, S.H.~Man, and L. Rolen, {\it Unimodality of Ranks and a Proof of Stanton's Conjecture}, preprint, \url{https://arxiv.org/abs/2209.12239}.
			
	\bibitem{Campbell} R. Campbell, \emph{Les int\'{e}grales eul\'{e}riennes et leurs applications. \'{E}tude
		approfondie de la fonction gamma}, Collection Universitaire de Math\'{e}matiques, XX, Dunod, Paris, 1966.
	
	\bibitem{CCM} G. Cesana, W. Craig, and J. Males, \textit{Asymptotic equidistribution for partition statistics and topological invariants}, preprint.
	
	\bibitem{CS} H. Cohen and F. Stromberg, \textit{Modular forms: a classical approach}, vol 179 of \textit{Graduate Studies in Mathematics}. American
	Mathematical Society, 2017.
	
	\bibitem{Chern} S. Chern, \textit{Nonmodular infinite products and a conjecture of Seo and Yee}, Adv. Math., 417 (2023), 108932.
	
	\bibitem{DeSalvo} S. DeSalvo and I. Pak, \textit{Log-concavity of the partition function}, Ramanujan J. {\bf 38} (2015), 1, 61--73.
	
			\bibitem{Dyson1944} F.~Dyson, \emph{Some guesses in the theory of partitions}, Eureka \textbf{8} (1944), 10--15.
			
	\bibitem{FMR} A. Folsom, J. Males, and L. Rolen, \textit{Equidistribution and partition polynomials}, Ramanujan J., to appear.
			
				\bibitem{GarvanTAMS} F. Garvan, \textit{New combinatorial interpretations of Ramanujan's partition congruences $\mod 5, 7,$ and $11$,} Trans. Amer. Math. Soc (305) no.\ 1 (1988), 47-77. 
			
		\bibitem{GORZ} M. Griffin, K. Ono, L. Rolen, and D. Zagier, \textit{Jensen polynomials for the {R}iemann zeta function and other
		sequences}, Proc. Natl. Acad. Sci. USA, {\bf 116} (2019), 23, 11103--11110.
	
	\bibitem{HardyRamanujan} G. Hardy and S. Ramanujan, \emph{Asymptotic formulae in combinatory analysis},
	Proc. London Math. Soc. Ser. 2 \textbf{17} (1918), 75--115.
	
	\bibitem{LZ} Z.-G. Liu and N. H. Zhou, \textit{Uniform asymptotic formulas for the 
	Fourier coefficients of the inverse of theta functions}, Ramanujan 
	J.  57 (2022), 1085--1123.
	
	\bibitem{NgoRhoades} H. Ngo and R. Rhoades. \textit{Integer Partitions, Probabilities and Quantum Modular Forms.}, Res. Math. Sci., \textbf{4}(2017). 

\bibitem{Wag} I. Wagner, \textit{Jacobi forms with {CM} and applications}, J. Number Theory, 246 (2023), 15--48.

\bibitem{Wagner} I. Wagner, \textit{On a new class of Laguerre-Pólya type functions with applications in number theory}, Pacific J. Math.320(2022), no.1, 177--192.

\bibitem{Zag} D. Zagier, \textit{The dilogarithm function}, Frontiers in number theory, physics, and geometry. II, Springer, Berlin (2007) 3--65
\end{biblist}
\end{bibsection}
\end{document}